\documentclass[12pt]{amsart}

\usepackage{amssymb,latexsym}
\usepackage{amscd}
\usepackage{verbatim,euscript}
\usepackage{fullpage}

\usepackage[colorlinks=true,linkcolor=blue,urlcolor=violet,citecolor=magenta]{hyperref}

\input {cyracc.def}
\numberwithin{equation}{section}

\font\tencyr=wncyr10 
\font\tencyi=wncyi10 
\font\tencysc=wncysc10 
\def\rus{\tencyr\cyracc}
\def\rusi{\tencyi\cyracc}
\def\rusc{\tencysc\cyracc}

\newtheorem{thm}{Theorem}[section] 

\newtheorem{lm}[thm]{Lemma}
\newtheorem{prop}[thm]{Proposition}

\theoremstyle{remark}
\newtheorem{rmk}[thm]{Remark}

\theoremstyle{definition}

\newtheorem*{ex-bn}{Example}
\newtheorem{df}{Definition}

\newenvironment{proof*}
{\noindent {\sl Proof.}\quad }{\hfill
$\square$}

\newcommand {\be}{{\mathfrak b}}

\newcommand {\fe}{{\mathfrak E}}

\newcommand {\g}{{\mathfrak g}}
\newcommand {\h}{{\mathfrak h}}

\newcommand {\el}{{\mathfrak l}}
\newcommand {\me}{{\mathfrak m}}

\newcommand {\te}{{\mathfrak t}}
\newcommand {\ut}{{\mathfrak u}}

\newcommand {\sln}{\mathfrak{sl}_n}

\newcommand {\spn}{\mathfrak{sp}_{2n}}
\newcommand {\sono}{\mathfrak{so}_{2n+1}}

\newcommand {\esi}{\varepsilon}
\newcommand {\ap}{\alpha}

\newcommand {\lb}{\lambda}
\newcommand {\vp}{\varphi}

\newcommand {\ca}{{\mathcal A}}

\newcommand {\cM}{{\mathcal M}}
\newcommand {\N}{{\mathfrak N}}
\newcommand {\co}{{\mathcal O}}
\newcommand {\cP}{{\mathcal P}}

\newcommand {\BN}{{\mathbb N}}
\newcommand {\BQ}{{\mathbb Q}}
\newcommand {\BZ}{{\mathbb Z}}

\newcommand {\VV}{{\mathsf V}}

\newcommand {\md}{/\!\!/}
\newcommand {\isom}{\stackrel{\sim}{\longrightarrow}}

\newcommand {\hot}{{\mathsf{ht}}}

\newcommand {\Lie}{{\mathrm{Lie\,}}}
\newcommand {\Ker}{{\mathrm{Ker\,}}}

\newcommand {\rk}{{\mathsf{rk\,}}}
\newcommand {\spe}{{\mathsf{Spec\,}}}

\newcommand {\tri}{\mathfrak{sl}_2}
\newcommand {\GR}[2]{{\textrm{{\bf #1}}}_{#2}}

\newcommand {\ov}{\overline}

\newcommand {\beq}{\begin{equation}}
\newcommand {\eeq}{\end{equation}}

\newcommand{\vts}{\VV_{\theta_s}}

\newcommand{\eus}{\EuScript}

\newcommand {\bbk}{\mathbb C}

\newfam\Bbbfam\newfam\eufam\newfam\eusfam%
\font\Bbbfont=msbm10 scaled 1200%
\font\olala=msam10 scaled 1200%
\font\Bbbsmallfont=msbm8%
\textfont\Bbbfam=\Bbbfont\scriptfont\Bbbfam=\Bbbsmallfont

\def\square{\hbox {\olala\char"03}}

\begin{document}
\setlength{\parskip}{2pt plus 4pt minus 0pt}
\hfill {\scriptsize May 23, 2012} 
\vskip1ex

\title{Invariant Theory of little adjoint modules}
\author[D.~Panyushev]{Dmitri I. Panyushev}
\address[]{Independent University of Moscow,
Bol'shoi Vlasevskii per. 11, 119002 Moscow, \ Russia
\hfil\break\indent
Institute for Information Transmission Problems, B. Karetnyi per. 19, Moscow 127994
}
\email{panyushev@iitp.ru}
\maketitle

\section{Introduction}
\label{sect:intro}

\noindent
Let $G$ be a complex simple Lie group having roots of different length.
Fix a triangular decomposition of $\g=\Lie G$ and the relevant objects (simple roots, 
dominant weights, etc.). In particular, let $\Delta$ be the set of all roots and
$\theta_s$ the short dominant root.
The simple $G$-module with highest weight $\theta_s$, denoted 
$\vts$, is said to be {\it little adjoint}.
There are two series of little adjoint representations (associated with  $G=Sp_{2n}$ or 
$SO_{2n+1}$)
and two sporadic cases (associated with $\GR{F}{4}$ and $\GR{G}{2}$).
We give a uniform presentation of 
invariant-theoretic properties of the little adjoint representations.
Most of these properties follows from known classification results in Invariant Theory. But
our intention is to provide conceptual proofs whenever possible. We also notice a new phenomenon; namely, a  relationship between $\vts$ and the adjoint representation of
certain simple subalgebra of $\g$.

Let  $\Pi_s$ be the set of short simple roots 
and $W(\Pi_s)$ the subgroup of the Weyl group
$W$ that is generated by the "short" simple reflections.
Let $\vts^0$ be the zero weight space of $\vts$. We prove that $\dim\vts^0=\#(\Pi_s)$,
$N_G(\vts^0)/Z_G(\vts^0)\simeq
W(\Pi_s)$, and the restriction homomorphism $\bbk[\vts]\to \bbk[\vts^0]$ induces an isomorphism
$\bbk[\vts]^G\simeq \bbk[\vts^0]^{W(\Pi_s)}$. This implies that $\bbk[\vts]^G$ is a polynomial
algebra, of Krull dimension $\#(\Pi_s)$, and 
the quotient morphism $\pi_G: \vts\to \vts\md G=\spe(\bbk[\vts]^G)$ is equidimensional.
If $v\in \vts^0$ is generic, then the stabiliser $G_v$ is  connected and semisimple,
and the root system of $G_v$ consists of all long roots in $\Delta$.
We also show that the orbit of highest weight vectors in $\vts$ is of dimension 
$2\mathsf{ht}(\theta_s)$ and $\dim\vts=(h+1){\cdot}\#(\Pi_s)$, where $h$ is the Coxeter 
number of $G$.

Let $\g(\Pi_s)$ be the semisimple subalgebra of $\g$ whose set of simple roots is $\Pi_s$.
Then $\rk\g(\Pi_s)=\#(\Pi_s)$ and $W(\Pi_s)$ is just the Weyl group of $\g(\Pi_s)$.
We give a conceptual explanation for the fact that $\Pi_s$ is a connected subset on the Dynkin diagram, so that $\el:=\g(\Pi_s)$
is actually simple.
There is a connection between $\vts$ and the adjoint representation of the group
$L=G(\Pi_s)$. Namely, $\el$ can naturally be regarded as a submodule of $\vts$
that contains $\vts^0$, and the restriction homomorphism $\bbk[\vts]\to \bbk[\el]$ induces an 
isomorphism $\bbk[\vts]^G\simeq \bbk[\el]^{L}$. Using the well-known properties of the adjoint
representation \cite{ko63}, we then prove that the null-cone $\N(\vts):=\pi_G^{-1}(\pi_G(0))$
is an irreducible complete intersection and $\vts$ admits a Kostant-Weierstrass section
(see Section~\ref{sect:null} for details).
All these results are proved conceptually. 

Let $\N(\el)$ denote the set of nilpotent elements in $\el$.
If $\co\subset \N(\el)$ is an $L$-orbit, then $G{\cdot}\co$ is a $G$-orbit in
$\N(\vts)$.  
There is a striking relation between the set of  $L$-orbits in $\N(\el)$
and the set of $G$-orbits in $\N(\vts)$, which is proved case-by-case.
The assignment $\co \mapsto G{\cdot}\co$ sets up a bijection between these two sets;
moreover,   if $\co\ne \{0\}$, then $\dim G{\cdot}\co /\dim\co =h/h_s$, where $h_s$ is the Coxeter number of $\el$. 
Using a relation of Coxeter elements, we conceptually prove  that $h/h_s\in\BN$.

In the Section~\ref{sect:remarksl}, we shortly discuss more advanced topics related to 
$\vts$ that are dealt with in \cite{short04,selecta}.

\subsection*{Main notation}
Throughout, $G$ is a connected simply-connected simple algebraic group with $\Lie G=\g$.
Fix a triangular decomposition $\g=\ut\oplus\te\oplus\ut^-$.  Then

{\bf --} \ $\Delta$ is the root system of $(\g,\te)$, $h$ is the Coxeter
number of $\Delta$, and
$W$ is the Weyl group. 

{\bf --} \ $\Delta^+$  is the set of positive roots corresponding to
$\ut$, $\theta$ is the highest root in $\Delta^+$,
and $\rho=\frac{1}{2}\sum_{\mu\in\Delta^+}\mu$. 

{\bf --} \ $\Pi=\{\ap_1,\ldots,\ap_n\}$ is the set of simple roots in $\Delta^+$
and $\vp_i$ is the fundamental weight corresponding to $\ap_i$.
If $\gamma\in\Delta$ and $\gamma=\sum_{i=1}^n c_i\ap_i$, then $\hot(\gamma)=\sum_i c_i$ is the {\it height\/} of $\gamma$.

{\bf --} \ ${\te}^*_{\BQ}$ is the $\BQ$-vector subspace of $\te^*$ generated by
the lattice of integral weights 
and $(\ \vert\ )$ is the $W$-invariant positive-definite 
inner product on ${\te}^*_{\BQ}$ induced by the Killing form on $\g$.  
As usual, $\mu^\vee=\frac{2\mu}{(\mu\vert\mu)}$ is the coroot
for $\mu\in \Delta$ and $\Delta^\vee=\{\mu^\vee\mid \mu\in\Delta\}$ is the dual root
system.

{\bf --} \ If $\lb$ is a dominant weight, then $\VV_\lb$ stands for the simple
$G$-module with highest weight $\lb$.

\noindent
For $\ap\in\Pi$, we let $r_\ap$ denote the corresponding simple reflection in $W$. 
If $\ap=\ap_i$, then we also write
$r_{\ap_i}=r_i$. 
The length function on $ W$ with respect to  $r_1,\dots,r_n$ is 
denoted by $\ell$. For any $w\in W$, we set
$\eus N(w)=\{\gamma\in\Delta^+ \mid w(\gamma) \in -\Delta^+ \}$.
It is standard that $\#\eus N(w)=\ell(w)$.

{\bf --} \ the linear span of a subset $M$ of a vector space is denoted by $\langle M\rangle$.

Our main reference on Invariant Theory is \cite{VP}.

{\small{\bf Asknowledgements.} I would like to thank the anonymous referee for several
helpful remarks and suggestions.}

\section{First properties}
\label{sect:first}

\noindent
Let $\g$ be a simple Lie algebra having two root lengths. We use subscripts
`s` and `l` to mark objects related to short and long roots, respectively.
For instance, $\Delta^+_s$ is the set of short positive roots, $\Delta=\Delta_s\sqcup
\Delta_l$, and $\Pi_s=\Pi\cap \Delta_s$. 
Recall that $\Delta_l=W{\cdot}\theta$, $\Delta_s=W{\cdot}\theta_s$, and 
$(\theta\vert\theta)/(\theta_s\vert\theta_s)=2$ or 3. 

Let $W_l$ be the subgroup of $W$ generated by
$r_\gamma$, where $\gamma\in\Delta^+_l$.
Let $W(\Pi_s)$ be the subgroup of $W$ generated by $r_\ap$, where $\ap\in\Pi_s$.
Then $W(\Pi_s)$ is a parabolic subgroup of $W$ in the sense of the theory of
Coxeter groups.

\begin{prop}    \label{semidir} \leavevmode\par
\begin{itemize}
\item[\sf (i)] \ $W(\Pi_s)=\{w\in W \mid w(\Delta^+_l)\subset \Delta^+_l\}$.
\item[\sf (ii)] \ 
$W\simeq W(\Pi_s)\ltimes W_l$.
\end{itemize}
\end{prop}\begin{proof}
(i) Obviously, $r_\ap(\Delta^+_l)\subset \Delta^+_l$ for any $\ap\in\Pi_s$.
Hence $W(\Pi_s)\subset\{w\in W \mid w(\Delta^+_l)\subset \Delta^+_l\}$.
On the other hand, if $w(\Delta^+_l)\subset \Delta^+_l$ and
$w=w'r_\ap$ is a reduced decomposition, then $\eus N(w)\subset\Delta^+_s$ and 
the equality $\eus N(w)=r_\ap(\eus N(w'))\cup\{\ap\}$
shows that $\ap$ is necessarily short. So,
we can argue by induction on $\ell(w)$.

(ii) Clearly, $W_l$ is a normal subgroup of $W$, and $W_l\cap W(\Pi_s)=1$ by part (i).
Therefore, it suffices to prove that the product mapping 
$W(\Pi_s)\times W_l\to W$ is onto. We argue by induction on the length of $w\in W$.
Suppose $w\not\in W(\Pi_s)$ and
$w=w_1r_\beta w_2\in W$, $\beta\in\Pi_l$, is a reduced decomposition.
Then $w=w_1w_2r_{\beta'}$, where $\beta'=w_2^{-1}(\beta)\in\Delta_l$,
and $\ell(w_1w_2)<\ell(w)$. 
That is, all long simple reflections occurring in an expression for $w$
can eventually be moved up to the right. 
\end{proof}%
Fix some notation, which applies to an arbitrary $\g$-module $V$.
Write ${\cP}(\VV)$ for the set of all weights of $\VV$. 
For instance, $\cP(\g)=\Delta\cup\{0\}$.
Let $\VV^\mu$ denote the
$\mu$-weight space of $\VV$ and $m_\VV(\mu)=\dim\VV^\mu$. If $\VV=\VV_\lb$, 
then the multiplicity is denoted by $m_\lb(\mu)$. 

\begin{prop}[cf. \protect{\cite[Prop.\,2.8]{ya-tg}}]
\label{shortchar}  \leavevmode\par
\begin{itemize}
\item[\sf (i)] \  $\dim\VV_{\theta_s}=(h+1)m_{\theta_s}(0)$;
\item[\sf (ii)] \  $m_{\theta_s}(0)=\# \Pi_s$;
\item[\sf (iii)] \ $\VV_{\theta_s}$ is an orthogonal $G$-module.
\end{itemize}
\end{prop}
\begin{proof} (i)
It is clear that ${\cP}(\vts)=\Delta_s\cup\{0\}$ and 
$m_{\theta_s}(\ap)=1$ for all $\ap\in\Delta_s$. Applying Freudenthal's weight
multiplicity formula  \cite[3.8, Proposition~D]{sam} 
to $m_{\theta_s}(0)$, we obtain  
\[
(\theta_s+2\rho|\theta_s)m_{\theta_s}(0)=2\sum_{\ap\in\Delta^+}\sum_{t\ge 1}
m_{\theta_s}(t\ap)(t\ap|\ap)=2\sum_{\ap\in\Delta_s^+}m_{\theta_s}(\ap)(\ap|\ap)
=2\sum_{\ap\in\Delta_s^+}(\ap|\ap) \ .
\]
Whence 
\[
(1+(\rho|\theta_s^\vee))m_{\theta_s}(0)=2{\cdot}\#\Delta_s^+=
\#\Delta_s=\dim\VV_{\theta_s}-m_{\theta_s}(0) \ .
\]
As $\theta_s^\vee$ is the highest root in the dual root system
$\Delta^\vee$, we have
$(\rho|\theta_s^\vee)=h-1$.

(ii) By part (i), we have $m_{\theta_s}(0)=
\displaystyle \frac{\dim\VV_{\theta_s}-m_{\theta_s}(0)}{h}=\frac{\#\Delta_s}{h}$.
Let $c\in W$ be a Coxeter element associated with $\Pi$. It is known
that each orbit of $c$ in $\Delta$ has cardinality $h$ and 
the number of orbits consisting of short roots is equal to $\#(\Pi_s)$,
see \cite[ch.VI,\,\S\,1,\,Prop.\,33]{bour}. 
Hence $\#\Delta_s=h{\cdot}\#\Pi_s$.

(iii) Since $\cP(\vts)=-\cP(\vts)$ and $m_{\theta_s}(\mu)=m_{\theta_s}(-\mu)$ for all 
$\mu\in\cP(\vts)$, 
we conclude that $\VV_{\theta_s}$ is self-dual.
Furthermore, because $\VV_{\theta_s}^0\ne 0$, it cannot be symplectic.
\end{proof}%

\begin{rmk}
It was shown by Zarhin \cite{zarhin} that
$(h+1)\dim\VV^0\le \dim \VV$ for any $\g$-module $\VV$.
Moreover, analysing his proof, one  readily concludes that the equality
can happen only if each nonzero weight of $\VV$ is a root, i.e.,
$\VV$ is either $\g=\VV_\theta$ or $\VV_{\theta_s}$. Thus, the adjoint and little adjoint
modules are distinguished by the condition that the ratio $\dim\VV/\dim\VV^0$ attains 
the minimal possible value.
\end{rmk}%
For any $\mu\in\Delta$, set $\Delta(\mu)=\{\gamma\in\Delta\mid (\gamma\vert\mu)\ne 0\}$.
Consider the partition of this set according to the sign of roots and of
the scalar product:
\[
   \Delta(\mu)=\Delta(\mu)^+_{>0}\sqcup \Delta(\mu)^+_{<0}
\sqcup \Delta(\mu)^-_{>0}\sqcup \Delta(\mu)^-_{<0} \ .
\]
Here $\Delta(\mu)^+_{>0}=\{\gamma\in\Delta^+\mid (\gamma\vert\mu)> 0\}$, and likewise for the
other subsets. \\
Since $\Delta(\mu)^+_{>0}=-\Delta(\mu)^-_{<0}$ and
$\Delta(\mu)^+_{<0}=-\Delta(\mu)^-_{>0}$, we obtain
\begin{equation}  \label{ravno}
  \#\Delta(\mu)^+=\#\Delta(\mu)_{>0} \ .
\end{equation}
Let $\eus C(\lb)$ denote the closure of the $G$-orbit of highest weight vectors
in $\VV_\lb$.

\begin{prop}  \label{dim-Cs}  \leavevmode\par
\begin{itemize}
\item[\sf (i)] \  If \ $\ap\in\Pi_s$, then \ $\#(\Delta(\ap)^+_{>0})=\hot(\theta_s)$
and \ $\#(\Delta(\ap)^+_{<0})=\hot(\theta_s)-1$;
\item[\sf (ii)] \ $\dim \eus C(\theta_s)=2\hot(\theta_s)$.
\end{itemize}
\end{prop}
\begin{proof}
(i) If $\ap$ is simple, then $r_\ap\bigl(\Delta(\ap)^+_{>0}\setminus\{\ap\}\bigr)=
\Delta(\ap)^+_{<0}$. Hence either of the two equalities implies the other.
Set $d_\ap=\#\bigl(\Delta(\ap)^+_{>0}\bigr)$. Then \ $\#\Delta(\ap)^+=2d_\ap-1$. 
To compute $d_\ap$, we look
at these subsets for $\theta_s$. Here
\[
\Delta(\theta_s)^+_{>0}=\Delta(\theta_s)_{>0}=\Delta(\theta_s)^+ .
\]
Set $\sigma =\frac{1}{2}\sum_{\gamma\in\Delta^+}\gamma^\vee$. Then $(\sigma\vert\gamma)=\hot(\gamma)$
for any $\gamma\in\Delta$. On the other hand, if $\gamma\in\Delta^+\setminus\{\theta_s\}$, 
then $(\gamma^\vee|\theta_s)\in\{0,1\}$. Therefore
\begin{multline*}
\hot(\theta_s)=(\sigma\vert\theta_s)=
\frac{1}{2}\bigl(\#(\Delta(\theta_s)^+_{>0})+1\bigr)= \\
\frac{1}{2}\bigl(\#(\Delta(\theta_s)_{>0})+1\bigr)=
\frac{1}{2}\bigl(\#(\Delta(\ap)_{>0})+1\bigr)=
\frac{1}{2}\bigl(\#(\Delta(\ap)^+)+1\bigr)=d_\ap \ .
\end{multline*}
In the last line, we have used Eq.~\eqref{ravno} with $\mu=\ap$ and
the fact that $\ap$ and $\theta_s$ are $W$-conjugate.

(ii) Let $v\in\vts$ be a highest weight vector. Then 
\[
\dim G{\cdot}v=1+\dim U_-{\cdot}v=1+\#\{\gamma\in \Delta^+\mid
\theta_s-\gamma\in \cP(\vts)\}=
1+\#\{\gamma\in\Delta^+ \mid (\gamma|\theta_s)>0\} \ .
\]
According to the proof of part (i), the last expression is equal to $2\hot(\theta_s)$.
\end{proof}
\begin{rmk}   \label{dual-cox}
Let $h^*(\Delta)$ denote the {\it dual Coxeter number\/} of $\Delta$.
By definition,  $h^*(\Delta)=1+(\rho|\theta^\vee)$. Notice that
$\theta^\vee$ is the short dominant root in $\Delta^\vee$
and $(\rho\vert\theta^\vee)$ is the height of $\theta^\vee$
in $\Delta^\vee$. Therefore, $h^*(\Delta^\vee)=1+(\sigma |\theta_s)=1+\hot(\theta_s)$.
This also means that $\dim \eus C(\theta_s)=2h^*(\Delta^\vee)-2$.
This can be compared with the well-known result that $\dim \eus C(\theta)=2h^*(\Delta)-2$.
\end{rmk}

\section{Generic stabilisers and the algebra of invariants}
\label{sect:invariants}

\noindent
Set $\h:=\te\oplus(\underset{\mu\in\Delta_l}{\oplus}\g^\mu)\subset \g$.
Obviously, it is a Lie subalgebra of $\g$.
Let $H$ denote the connected subgroup of $G$ with Lie algebra $\h$.
Then $\rk H=\rk G$ and $H$ is semisimple. The Weyl group
of $(\h,\te)$ is $W_l$.
Let $\pi_G:\vts\to \vts\md G:=\spe\bbk[\vts]^G$ denote the quotient morphism. 
For any $\mu\in \Delta$, fix a nonzero element $e_\mu\in\g^\mu$. 

\begin{thm}    \label{sop} \leavevmode\par
\begin{itemize}
\item[\sf (i)] \ $\vts^0=(\vts)^H$;
\item[\sf (ii)] \ $G{\cdot}\VV_{\theta_s}^0$ is dense in $\vts$ and
$\h$ is a generic stationary subalgebra for $(G:\vts)$;
\item[\sf (iii)] \  $\bbk[\vts]^G\simeq \bbk[\vts^0]^{W(\Pi_s)}$;
\item[\sf (iv)] \  $\bbk[\vts]^G$ is a polynomial algebra and $\pi_G$ is equidimensional.
\item[\sf (v)] \  All the fibres of $\pi_G$ are of dimension \ $h{\cdot}\dim \vts^0=
h{\cdot}\#\Pi_s$.
\end{itemize}
\end{thm}
\begin{proof}  \leavevmode\par
(i) \ Since $T\subset H$, we have $\VV_{\theta_s}^0\supset (\VV_{\theta_s})^H$.
On the other hand, if $\mu\in \Delta_l$,  then $e_\mu{\cdot}\VV_{\theta_s}^0=0$.

(ii) \ 
By Elashvili's Lemma \cite[\S 1]{alela}, $G{\cdot}\VV_{\theta_s}^0$ is dense in 
$\VV_{\theta_s}$ if and only if there is $x\in \VV_{\theta_s}^0$ such that 
$\g{\cdot}x+\VV_{\theta_s}^0=\VV_{\theta_s}$.
To prove the last equality, take any $\mu\in\Delta_s$ and
consider $e_\mu$ as the operator 
$\tilde e_\mu: \VV_{\theta_s}^0\to \VV_{\theta_s}^{\mu}$.
If it were zero operator, then all such operators
would be zero, since $W{\cdot}\mu=\Delta_s$.
That is, we would obtain $\VV_{\theta_s}^0=(\VV_{\theta_s})^G$, which is absurd.
Hence $\Ker \tilde e_\mu $ is a hyperplane in $\vts^0$ for any $\mu\in\Delta_s$.
It follows that, for any $x\in \vts^0\setminus \displaystyle
\bigcup_{\mu\in\Delta_s} \Ker \tilde e_\mu$,  we have $\g_x=\h$ and $\g{\cdot}x=
\oplus_{\mu\ne 0} \vts^\mu$.

(iii) \ By part (ii), if $x\in\vts^0$ is generic, then the
identity component of $G_x$ is $H$.
Since the orbit $G{\cdot}x$ is closed for any $x\in\vts^0=(\vts)^H$ \cite{ko63}, we may 
apply a generalization of the Chevalley restriction theorem
\cite[Theorem\,5.1]{lr}. It claims that
\[
   \bbk[\vts]^G\simeq \bbk[\vts^0]^{N_G(H)/H} \ .
\]
Since $N_G(H)/H=N_G(T)H/H\simeq N_G(T)/N_H(T)\simeq W/W_l\simeq W(\Pi_s)$, we are done.

(iv) \ Since $G$ is connected and $W(\Pi_s)$ is finite, this follows from (iii)
and  \cite{fin}.

(v) \ This follows from (iv) and Prop.~\ref{shortchar}.
\end{proof}%
\begin{rmk}   \label{rem-CSS}
a) The $G$-module $\vts$  is {\it stable}, i.e., the union of closed $G$-orbits 
contains a dense open subset of $\vts$. This follows from \cite{po}, since a generic stationary  subalgebra $\h$ is reductive; 
or, from \cite{luna72}, since $\vts$ is an orthogonal $G$-module. 
The stability can also be derived from the equality $\ov{G{\cdot}\vts^0}=\vts$
and the fact that each $G$-orbit meeting the zero weight space 
is closed \cite[Remark 11 on p.\,354]{ko63}.
\\
b)  The equality $\vts=\vts^0\oplus \g{\cdot}x$, which holds for almost all $x\in\vts^0$,
means that $\vts^0$ is a {\it Cartan subspace\/} of $\vts$ in the sense of \cite{polar} and
\cite{R&S}. 
\end{rmk}
By Theorem~\ref{sop}(ii), the identity component of a generic stabiliser is conjugate to $H$.
Below, we prove that generic stabilisers are connected, i.e., $H$ itself is a generic stabiliser.

In what follows, $(\ ,\ )_s$ stands for a nonzero $G$-invariant symmetric bilinear form on 
$\vts$. 
As we have proved, $\eus H_\mu=:\Ker\tilde e_\mu$ is a hyperplane in
$\vts^0$ for any $\mu\in\Delta_s$. 
Our next goal is to study the hyperplane arrangement  
obtained in this way. For each $\mu\in\Delta_s$, fix a nonzero vector $v_\mu\in V_{\theta_s}^\mu$.
Let $\{e_\mu, h_\mu, e_{-\mu}\}$ be a standard $\tri$-triple in $\g$ corresponding
to $\mu\in \Delta^+_s$. In particular, $\mu(h_\mu)=2$. 
Set $\tri(\mu)=\langle e_\mu, h_\mu, e_{-\mu} \rangle$.

\begin{prop}    \label{konfig} \leavevmode\par
\begin{itemize}
\item[\sf (i)] \ For any $\mu\in\Delta^+_s$, we have
$\eus H_\mu=\eus H_{-\mu}$, 
and the restriction of\/ $(\ ,\ )_s$ to $\eus H_\mu$ is 
non-degenerate;
$\langle e_\mu{\cdot} v_{-\mu}\rangle=
\langle e_{-\mu}{\cdot}v_\mu\rangle$, and it is the orthogonal complement to
$\eus H_\mu$ in $\VV_{\theta_s}^0$.
\item[\sf (ii)] \ Suppose that  $\gamma,\mu\in \Delta_s$ and $\nu:=\gamma-\mu\in \Delta_l$. 
Then $\eus H_\gamma=\eus H_\mu$.
\end{itemize}
\end{prop}
\begin{proof}
(i) \ We have $e_{-\mu}{\cdot}(e_\mu{\cdot} v_{-\mu})=-h_{\mu}{\cdot}v_{-\mu}=
\mu(h_{\mu}){\cdot}v_{-\mu}=2v_{-\mu}\ne 0$. Also,
$h_{\mu}{\cdot}(e_\mu{\cdot} v_{-\mu})=[h_\mu,e_\mu]{\cdot}v_{-\mu}
+e_{\mu}(h_{\mu}{\cdot}v_{-\mu})=0$.
It follows from these equalities and the $\tri$-theory that 
$e_{\mu}{\cdot}(e_\mu{\cdot} v_{-\mu})\ne 0$.
Thus, $\langle v_{-\mu}, e_\mu{\cdot} v_{-\mu}, e_{\mu}{\cdot}(e_\mu{\cdot} v_{-\mu})
\rangle$ is a 3-dimensional simple $\tri(\mu)$-module.
Since $e_{\mu}{\cdot}(e_\mu{\cdot} v_{-\mu})$ is proportional to $v_\mu$, we obtain 
$\langle e_\mu{\cdot} v_{-\mu}\rangle=\langle e_{-\mu}{\cdot}v_\mu\rangle$.

Since $(e_\mu{\cdot} v_{-\mu},e_{-\mu}{\cdot} v_{\mu})_s=
-(e_{-\mu}{\cdot}(e_\mu{\cdot} v_{-\mu}), v_{\mu})_s\ne 0$, the line
$\langle e_\mu{\cdot} v_{-\mu}\rangle$ is not isotropic.
Finally, $0=(\eus H_\mu, e_\mu{\cdot} v_{-\mu})_s$. Hence $\eus H_\mu=
\langle e_\mu{\cdot} v_{-\mu}\rangle^\perp$. By the symmetry, we conclude that
$\eus H_\mu=\eus H_{-\mu}$.

(ii) \ Up to a nonzero factor, we have $[e_\mu,e_\nu]=e_\gamma$.
Consequently, for any $v\in \vts^0$,  
\[
e_\gamma{\cdot}v=[e_\mu,e_\nu]{\cdot}v=
(e_\mu e_\nu-e_\nu e_\mu){\cdot}v=
-e_\nu{\cdot}(e_\mu{\cdot}v) \ .
\]
This readily implies that $\Ker\tilde e_\gamma=\Ker \tilde e_\mu$, i.e.,
$\eus H_\gamma = \eus H_\mu$.
\end{proof}%
Let $\g(\Pi_s)$ be the Lie subalgebra of $\g$ generated by
$\g^{\pm\ap}$, $\ap\in\Pi_s$. Then $\g(\Pi_s)$ is semisimple and
its root system is $\Delta(\Pi_s):=\Delta\cap\mathbb Z\Pi_s$.
It is easily seen that $\g(\Pi_s)$ is the commutant of a Levi subalgebra of $\g$. 
Obviously, $\Pi_s$ is a set of simple roots for $\g(\Pi_s)$ and 
$W(\Pi_s)$ is the Weyl group of $\g(\Pi_s)$. 
Notice that $\Delta(\Pi_s)$ is a proper subset of $\Delta_s$.
Let $G(\Pi_s)$ be the connected semisimple subgroup of $G$
with Lie algebra $\g(\Pi_s)$. 

\begin{lm}
$\vts\vert_{G(\Pi_s)}$ contains the adjoint representation of 
$G(\Pi_s)$. If\/ $\widetilde\VV$ is any other simple $G(\Pi_s)$-submodule of\/ $\vts$, then
$\widetilde\VV\cap \vts^0=\{0\}$.
\end{lm}
\begin{proof}
Consider the subspace 
\[
  \vts^0\oplus(\underset{\mu\in\Delta(\Pi_s)}{\oplus}\vts^\mu) \subset\vts \ .
\]
It is clear that it is a $G(\Pi_s)$-submodule of $\vts$, and using Proposition~\ref{shortchar}(ii) one readily concludes that it
is isomorphic to $\g(\Pi_s)$. 
The complementary $G(\Pi_s)$-submodules are \quad 
$\bigoplus_{\mu\in \Delta^+_s\setminus \Delta(\Pi_s)} \vts^\mu$ \ and \ 
$\bigoplus_{\mu\in \Delta^-_s\setminus \Delta(\Pi_s)} \vts^\mu$ .
\end{proof}

We shall  identify the $G(\Pi_s)$-module $\g(\Pi_s)$ with the above 
submodule of $\vts$.
Consider the commutative diagram
\begin{equation}  \label{diagram}
\begin{CD}
\vts^0  
          @>>> \g(\Pi_s)                @>>> \vts  \\
@VV{\pi_{W(\Pi_s)}}V                    @VV\pi_{G(\Pi_s)}V        @VV{\pi_G}V    \\
\vts^0\md W(\Pi_s) @>g>> \g(\Pi_s)\md G(\Pi_s)  @>f>> \vts\md G 
\end{CD}
\end{equation}
Here the arrows in the top row are embeddings and the vertical arrows are
the quotient morphisms. Recall that the  $W(\Pi_s)$-action on $\vts^0$
arises from the identification $W(\Pi_s)\simeq W/W_l$. 
The existence of $g$ follows from the fact that $W(\Pi_s)$
can also be regarded as a subquotient of $G(\Pi_s)$.
By Theorem~\ref{sop}(iii), the composition $f{\circ} g$ is an isomorphism.
Furthermore, $g$ is finite and surjective, and $f$ is surjective.
Therefore,  both $f$ and $g$ are isomorphisms.
From this we deduce that action of $W(\Pi_s)$ on $\vts^0$
is isomorphic to the reflection representation of
the Weyl group of $G(\Pi_s)$ on the Cartan subalgebra in $\g(\Pi_s)$.

From these  properties of diagram~\eqref{diagram} we derive some further
conclusions. 

\begin{prop}  \label{simple}  \leavevmode\par

1. The Lie algebra $\g(\Pi_s)$ is simple. 

2. The generic stabiliser for the action $(G:\vts)$ is connected (and equal to $H$). 

3. The set of  hyperplanes $\{\eus H_\mu\}_{\mu\in\Delta^+_s}$
coincides with $\{\eus H_\mu\}_{\mu\in\Delta(\Pi_s)^+}$. All the hyperplanes in
the last set are different. 
\end{prop}
\begin{proof}
1. As $\vts$ is a simple orthogonal $G$-module, $\bbk[\vts]^G$
has a unique invariant of degree 2. On the other hand, the number of linearly independent
invariants of degree 2 in
$\bbk[\g(\Pi_s)]^{G(\Pi_s)}$ equals the number of simple factors of $\g(\Pi_s)$.
Because the mapping $f$ in \eqref{diagram}
is an isomorphism, $\g(\Pi_s)$ must be simple.

2. Let $G_\ast$ be a generic stabiliser for $(G:\vts)$.
Without loss of generality, assume that $G_\ast\supset H$.
If $G_\ast\ne H$, then the finite group
$W(\Pi_s)\simeq N_G(H)/H$ acts on $\vts^0$ non-effectively.
But we know from diagram~\eqref{diagram} that this is not the case.

3. The hyperplanes $\{\eus H_\mu\}_{\mu\in\Delta(\Pi_s)^+}$ are just the reflecting hyperplanes
for the reflection representation of $W(\Pi_s)$. Therefore they are all different.
Take any $\eus H_\gamma$ with $\gamma\in\Delta^+_s\setminus \Delta(\Pi_s)^+$.
Then there is
a $w\in W$ such  that $w{\cdot}\gamma \in\Delta(\Pi_s)$. In view of 
Proposition~\ref{semidir}(ii), we may assume that $w\in W_l$. 
Write $w=r_{\beta_m}\ldots r_{\beta_1}$, where $\beta_i\in\Delta^+_l$.
Then we get a string of {\sl short\/}
roots $\gamma=\nu_0,\nu_1,\ldots,\nu_m=\mu$ such that $\nu_{i+1}-\nu_i\in\Delta_l$.
By Proposition~\ref{konfig}(ii),  $\eus H_{\nu_i}=\eus H_{\nu_{i-1}}$ for all $i$. Hence 
$\eus H_\gamma=\eus H_{w{\cdot}\gamma}$.
\end{proof}%
\begin{rmk}  
A case-by-case verification shows that for any $\gamma\in
\Delta^+_s\setminus \Delta(\Pi_s)^+$ there is a sole long root $\beta$ such that 
$\gamma-\beta\in \Delta(\Pi_s)$, i.e., there is a string, as above, with $m=1$.

$\bullet$ \ $\g=\spn$. Here $\Delta(\Pi_s)^+=\{\esi_i-\esi_j\mid 1\le i< j\le n\}$,
$\Delta^+_s=\{\esi_i\pm \esi_j\mid 1\le i< j\le n\}$, and
$\Delta_l^+=\{2\esi_i \mid 1\le i\le n\}$. If $\gamma=\esi_k+\esi_l$ $(k<l)$, then
$\esi_k+\esi_l=(\esi_k-\esi_l)+2\esi_l$ is the required decomposition.

$\bullet$ \  $\g=\sono$. Here $\Delta(\Pi_s)^+=\{\esi_n\}$, $\Delta^+_s=
\{\esi_i \mid 1\le i\le n\}$,
and $\Delta_l^+=\{\esi_i\pm \esi_j\mid 1\le i< j\le n\}$. 
If $\gamma=\esi_k$ $(k<n)$, then
$\esi_k=(\esi_k-\esi_n)+\esi_n$ is the required decomposition.

The cases of $\GR{F}{4}$ and $\GR{G}{2}$ are left to the reader.
\end{rmk}

\section{The null-cone and Kostant-Weierstrass section}
\label{sect:null}

\noindent
In this section, we compare invariant-theoretic properties of the representations
$(G:\vts)$ and $(G(\Pi_s):\g(\Pi_s))$.
\begin{df}
The simple Lie algebra $\g(\Pi_s)$ is called the {\it simple reduction\/} of the little 
adjoint representation $(G:\vts)$.
\end{df}

\noindent 
To a great extent, invariant-theoretic properties of $(G:\vts)$ are determined
by its simple reduction. We have already proved that 
$\g(\Pi_s)\md G(\Pi_s)\simeq \vts\md G$, and  further results are presented below. 
To simplify notation, we set $L=G(\Pi_s)$ and $\el=\g(\Pi_s)$. Recall that $\el$ is regarded
as an $L$-submodule of $\vts$.

Let $\N(\vts)$ and $\N(\el)$ denote the null-cones in $\vts$ and $\el$, respectively, 
i.e., $\N(\vts)=\pi^{-1}_G(\pi_G(0))$ and $\N(\el)=\pi_L^{-1}(\pi_L(0))$.
All elements of the null-cone are said to be {\it nilpotent}. 

\begin{thm}   \label{thm:irred}  \leavevmode\par
\begin{itemize}
\item[\sf (i)] \ the variety $\N(\vts)$ is irreducible;
\item[\sf (ii)] \ 
there is $e\in \N(\vts)$ such that $d\pi_G(e)$ is onto;
\item[\sf (iii)] \ the ideal of the variety $\N(\vts)$ in
$\bbk[\vts]$ is generated by the basic $G$-invariants.
\end{itemize}
\end{thm}
\begin{proof}
(i), (ii). It follows from diagram~\eqref{diagram} that $\N(\vts)\cap\el=\N(\el)$. It is also
known that $\N(\el)$ is irreducible and $\dim\N(\el)=\dim\el-\rk\el
=\dim\el-\dim\vts^0$ \cite{ko63}. 
Let $\N_1$ be an irreducible component of
$\N(\vts)$. Then $\dim \N_1=\dim\vts-\dim\vts^0$ and
\[
\dim\N_1\cap\el\ge \dim \N_1 +\dim\el-\dim\vts=\dim\N(\el) \ .
\]
It follows that $\N_1\cap\el=\N(\el)$, i.e., 
each irreducible component of  $\N(\vts)$ contains $\N(\el)$. 
By \cite{ko63}, there is $v\in\N(\el)$ such that $d\pi_L(v)$ is onto. 
It then follows from properties of diagram~\eqref{diagram} that
$d\pi_G(v)$ is onto as well. Hence $v$ is a smooth point of the
fibre $\pi^{-1}_G(\pi_G(0))$. Therefore, $v$ lies in a unique irreducible component
of $\N(\vts)$ and  $\N(\vts)$ is irreducible.

(iii) \ This follows from (i) and (ii) 
(cf. \cite[Lemma~4 on p.\,345]{ko63}).
\end{proof}

\begin{rmk}
a)  Using the Hilbert-Mumford criterion \cite[\S\,5]{VP} and the structure of weights of
$\vts$, one can give another proof of the irreducibility of 
$\N(\vts)$. 
\\  \indent
b)  We have proved that $\pi_G$ is equidimensional and the fibre 
$\pi^{-1}_G(0)=\N(\vts)$ is an irreducible reduced complete intersection. By a standard 
deformation argument, this implies that the same properties hold for all the  fibres
of $\pi_G$. 
\end{rmk}
An affine subspace $\ca$ of a $G$-module $\VV$ is called a {\it Kostant-Weierstrass
section\/} (KW-section, for short), if the restriction of the quotient morphism
$\pi:\VV\to \VV\md G$ to $\ca$ yields an isomorphism
$\pi\vert_\ca:\ca\isom\VV\md G$. See \cite[8.8]{VP} for details on KW-sections. 

\begin{thm}   \label{KW}
The \ $G$-module $\vts$ has a KW-section.
\end{thm}\begin{proof}
Let $e\in\N(\el)$ be an $L$-regular nilpotent element. 
Then $d\pi_L(v)$ is onto, and hence $d\pi_G(v)$ is onto.
Therefore $e$ is a smooth point of $\N(\vts)$.
Since $G{\cdot}e$ is conical, we can find a semisimple element $x\in\g$ such that
$x{\cdot}e=e$. Take an $x$-stable complement to $T_e(\N(\vts))$ in $\vts$. Call it $U$.
Then $e+U$ is a KW-section in $\vts$. A standard argument for the last claim
can be found in \cite[Prop.\,4]{R&S0} 
(see also \cite[8.8]{VP}).
\end{proof}

By Proposition~\ref{simple}(i), $\Delta(\Pi_s)$ is an irreducible (simply-laced)
root system. Therefore the Coxeter number of $\Delta(\Pi_s)$ is well-defined.
Write $h_s$ for this number.

\begin{prop}
Let $c\in W$ be a Coxeter element associated with $\Pi$. Then
$c^{h_s}\in W_l$ and $h/h_s\in \BN$.
\end{prop}\begin{proof}
By Proposition~\ref{semidir}, we can write $c=c_1c_2$, 
where $c_1\in W(\Pi_s)$ and $c_2\in W_l$. 
Furthermore, $c_1$ is a Coxeter element of $W(\Pi_s)$, and the semi-direct product structure of $W$ shows that 
$c^k=(c_1)^k c'_2$ for some $c'_2\in W_l$. Taking $k=h_s$ or $h$, we obtain both
assertions.
\end{proof}%
\begin{df}The integer $h/h_s$ is called the {\it transition factor\/}. 
\end{df}

\noindent
By our results for $(G:\vts)$ and well-known properties of simple Lie algebras,
we have
\[\begin{array}{rrr}
\bullet & \dim\vts=(h+1){\cdot}\#(\Pi_s),  &\dim\N(\vts)=h{\cdot}\#(\Pi_s); \\
\bullet & \dim\el=(h_s+1){\cdot}\#(\Pi_s), &\dim\N(\el)=h_s{\cdot}\#(\Pi_s);
\end{array}
\]
It follows that $\dim\N(\vts)/\dim\N(\el)$ equals the transition factor.
Actually, the relationship between these null-cones is much more precise
and mysterious!

\begin{thm}    \label{thm:nil-orb}
Let $\co$ be a nilpotent $L$-orbit in $\el$. 
The mapping $\co\to G{\cdot}\co$ sets up a bijection
between the sets of nilpotent orbits $\N(\el)/L$ and $\N(\vts)/G$.
Moreover, this mapping preserves the closure relation and \ 
$\displaystyle
\frac{\dim(G{\cdot}\co)}{\dim\co}=\frac{h}{h_s}
$ \ \ 
for any nonzero $\co\in \N(\el)/L$. 
\end{thm}\begin{proof}
Unfortunately, the proof relies on an explicit classification of 
orbits in $\N(\vts)$.
(It is would be great to have a conceptual explanation!) The four possibilities  
are gathered in Table~\ref{tab1}.

\begin{table}[h]
\begin{tabular}{c|lccc|ccc|c}
& $\ \g$  &  $\dim\vts$ & $\theta_s$ & $h$   & $\el=\g(\Pi_s)$ &  $h_s$ &  $\#(\N(\el)/L)$ 
& $\tilde\g$ \\ \hline
1\  & $\spn$  & $2n^2{-}n{-}1$ & $\vp_2$   & $2n$ &  $\sln$  & $n$    & $\#\mathsf{Par}(n)$ 
& $\mathfrak{sl}_{2n}$ \\
2\  & $\sono$  & $2n+1$  & $\vp_1$   & $2n$ &  $\tri$                  & $2$   & $2$
& $\mathfrak{so}_{2n+2}$ \\
3\  & $\GR{F}{4}$ & $26$ & $\vp_1$ & $12$ & $\mathfrak{sl}_3$ & $3$ & $3$ & $\GR{E}{6}$ \\
4\  & $\GR{G}{2}$  & $7$ & $\vp_1$ & $6$  &  $\mathfrak{sl}_2$ & $2$ & $2$
& $\mathfrak{so}_{8}$    
\end{tabular}
\vskip1ex
\caption{The little adjoint representations and their simple reductions}  \label{tab1}
\end{table}
\noindent 
The only non-trivial case is the first one. Here $\mathsf{Par}(n)$ stands for the set of all
partitions of $n$, and a classification of the nilpotent $Sp_{2n}$-orbits in
$\vts$ is obtained in \cite[\S\,3.2]{sek84}. 
\end{proof}

\begin{rmk}   \label{rmk:strange-eq}
A case-by-case inspection shows that $h/h_s=h-\mathsf{ht}(\theta_s)=
\mathsf{ht}(\theta)-\mathsf{ht}(\theta_s)+1$. Again, it would be
interesting to have an explanation for this.
\end{rmk}

\begin{rmk}  \label{rmk:pro-invol}
For items 1--3 in Table~\ref{tab1}, the little adjoint representation is the isotropy 
representation of a symmetric space of certain over-group $\tilde G$, i.e., it is related to an
involution of $\tilde\g=\Lie \tilde G$. The algebra $\tilde\g$ is indicated in the last column
of Table~\ref{tab1}. It is interesting to observe that in these cases the restricted root 
system of the symmetric variety $\tilde G/G$ is reduced and of type $\el$ 
(that is, of type $\GR{A}{n-1}$ for item~1, etc.). \quad
Item~4 is related to an automorphism of order $3$ of $\tilde\g=\mathfrak{so}_8$.
Therefore, a classification of nilpotent $G$-orbits in $\vts$ can also be obtained
via a method of Vinberg~\cite{vi79}.
\end{rmk}

\noindent
For an arbitrary $G$-module $\VV$,
set $\eus R_G(\VV)=\{v\in\VV \mid \dim G{\cdot}v \textrm{ is maximal}\}$.
It is a dense open subset of $\VV$. The elements of $\eus R_G(\VV)$ are usually
called {\it regular}.
Consider the quotient morphism
$\pi_{G,\VV}:\VV\to \VV\md G:=\spe\bbk[\VV]^G$. 
Set $\eus S_G(\VV)=\{v\in\VV \mid d\pi_{G,\VV}(v) \textrm{ is onto}\}$.
A classical result of Kostant \cite[Theorem\,0.1]{ko63} asserts that $\eus R_G(\g)=\eus S_G(\g)$.
Another proof  is given in \cite[\S\,1]{R&S0}. 

\begin{prop}   \label{prop:RiS}
We have $\eus R_G(\vts)=\eus S_G(\vts)$.
\end{prop}
\begin{proof}
1. First, we notice that $\eus R_G(\vts)\subset \eus S_G(\vts)$.
This is a consequence of Theorem~\ref{sop}, Remark~\ref{rem-CSS}(b), 
and \cite[Corollary 1]{R&S}. 
For, the theory developed in \cite{R&S} shows that the required inclusion always 
holds for the representations with a Cartan subspace. 

2. To prove the converse, we first note that 
$\eus R_G(\vts)\cap\N(\vts)= \eus S_G(\vts)\cap\N(\vts)$. For items~1--3 of Table~\ref{tab1},
this follows from \cite[Theorem~4]{sek84}.
Indeed, these items are related to involutions of a group $\tilde G$, and Sekiguchi's theorem
asserts that such an equality holds if and only if the restricted root system of
$\tilde G/G$ is reduced (cf. Remark~\ref{rmk:pro-invol}). The last item of Table~\ref{tab1} is easy.

In order to reduce the problem to nilpotent elements, we use Luna's slice theorem (see 
\cite[\S\,6]{VP}).
If $\ov{G{\cdot}v}\not\ni \{0\}$, then there exists a generalised Jordan decomposition
$v=s+n$, which means that $G{\cdot}s$ is closed ($s\ne 0$) and $\ov{G_{s}{\cdot}n} \ni \{0\}$. Without loss of generality, we may assume that $s\in \vts^0$.
Modulo trivial representations,
the slice representation $(G_s: N_s)$ 
associated with $s$ is the direct sum of little adjoint
representations for the simple components of $G_s$; and $n$ is a nilpotent element in $N_s$. It remains to observe that the slice theorem implies that
$v\in \eus R_G(\vts)\ \Leftrightarrow \  n\in \eus R_{G_s}(N_s)$ and
$v\in \eus S_G(\vts)\ \Leftrightarrow \  n\in \eus S_{G_s}(N_s)$.
\end{proof}%

\begin{rmk}
The null-cone $\N(\vts)$ is an irreducible complete intersection, and it follows from 
Theorem~\ref{thm:nil-orb} that
the complement of the
dense $G$-orbit in $\N(\vts)$ is of codimension $2h/h_s$, which is $\ge 4$. 
Therefore, $\N(\vts)$ is normal. 
Moreover, in this situation, the closure of any nilpotent $G$-orbit is normal!
Again, the only non-trivial case is item~1 in Table~\ref{tab1}.
For this case, the normality of all nilpotent orbit closures is proved in \cite[Theorem\,4]{ohta}.
\end{rmk}

\section{Further properties and remarks}
\label{sect:remarksl}

\noindent 
\subsection{}
There is a rich combinatorial theory for ideals of the Borel subalgebra
$\be=\te\oplus\ut$ in $\ut$, which is mainly due to Cellini and Papi (see e.g.
\cite[Sect.\,2]{short04} and references therein). In particular, there is a nice closed formula 
for the number of such ideals.
This formula has an analogue in the context of the
little adjoint representations.

Consider the $B$-stable space $\vts^+=\oplus_{\mu\in\Delta^+_s}\vts^\mu \subset \vts$.
Then there is a bijection between the $B$-stable subspaces of $\vts^+$ and 
the antichains in the poset $\Delta^+$ that consists of short roots
\cite[Prop.\,4.2]{short04}.
The common cardinality $K$ of these two sets is given as follows.
Let $m_1\le m_2\le\dots \le m_n$ be the exponents of $W$ and $l=\#\Pi_s$. Then
\[
      K=\prod_{i=1}^l \frac{h+m_i+1}{m_i+1} \ .
\]
For items 1--3 in Table~\ref{tab1}, i.e., if $(\theta\vert\theta)/(\theta_s\vert\theta_s)=2$,
there is a slightly different formula:
\[
    K=\prod_{i=1}^n \frac{g+m_i+1}{m_i+1} \ ,
\]
where $g=\#\Delta_s/n$, see \cite[Theorem\,5.5]{short04}.

\subsection{} For a graded $G$-module $\cM=\oplus_i \cM_i$ with $\dim\cM_i< \infty$, the graded character of $\cM$, $\mathsf{ch}_q(\cM)$, 
is the formal sum $\sum_i \mathsf{ch}(\cM_i) q^i \in \boldsymbol{\Lambda}[[q]][q^{-1}]$.
Here $\boldsymbol{\Lambda}$ is the character ring of finite-dimensional representations of
$G$. The graded character of
$\bbk[\N(\g)]$ was determined by Hesselink in 1980 \cite{wim2}. 
A similar formula exists for
$\mathsf{ch}_q(\bbk[\N(\vts)])$. This is a particular instance of the theory of
short Hall-Littlewood polynomials developed in \cite[Sect.\,5]{selecta}.

Let us define a $q$-analogue of a generalised partition function 
$\ov{\eus P}_q(\nu)$ by the expansion
\[
   \prod_{\mu\in\Delta^+_s} \frac{1}{1-qe^\mu}=\sum_{\nu}\ov{\eus P}_q(\nu)e^\nu .
\]
and for $\lb$ dominant, we set
\[
  \ov{\me}_{\lb}^{\mu}(q)=\sum_{w\in W} (-1)^{\ell(w)}\ov{\eus P}_q(w(\lb+\rho)-(\mu+\rho))  .
\]
Then  (see \cite[Prop.~5.6]{selecta})
\[
    \mathsf{ch}_q(\bbk[\N(\vts)])=\sum_{\lb \ \text{ dominant}} \ov{\me}_{\lb}^{0}(q)\,
    \mathsf{ch}\VV_\lb .
\]

\subsection{}
For any orthogonal $G$-module $\VV$, one can define a
subvariety of $\VV\times \VV$, which is called the {\it commuting variety} (of $V$). Namely, 
if $\eus K$ is the Killing form on $\g$ and $<\,,\,>$ is a $G$-invariant symmetric non-degenerate bilinear form on $\VV$, then we
consider the bilinear mapping
\[ 
    \vp: \VV\times \VV\to\g ,
\]  
where $\eus K (\vp (v_1,v_2), s):=<s{\cdot}v_1,v_2>$,
$s\in\g,\,v_1,v_2\in \VV$. 
By definition, $\fe(\VV):=\vp^{-1}(0)_{red}$ is the commuting variety.
One of the first questions is whether $\fe(\VV)$ is  irreducible.

{\bf Example}. If $\VV=\g$, then $\vp=[\ ,\ ]$ and $\fe(\g)$ is the usual commuting variety,
i.e., the set of pairs of commuting elements in $\g$. A classical result
of Richardson \cite{rich} asserts that $\fe(\g)$ is irreducible.
More generally, if $\g=\g_0\oplus\g_1$ is a $\BZ_2$-grading, then 
$\g_1$ is an orthogonal $G_0$-module and 
$\vp:\g_1\times\g_1\to\g_0$ is nothing but the usual Lie bracket.
However, the commuting  variety $\fe(\g_1)$ is not always irreducible \cite{fan}.

\begin{thm}   \label{thm:irr-com}
The commuting variety $\fe(\vts)$ is irreducible.
\end{thm}\begin{proof}
It would be pleasant to have a case-free argument, in the spirit of
Richardson's approach. But we can only provide  
a case-by-case proof, which runs as follows. There are four pairs 
$(G,\vts)$:

 $(Sp(\VV), \wedge^2_0 \VV)$; \ $(SO(\VV), \VV)$, $\dim \VV$ is odd; \ 
    $(\GR{F}{4}, \VV_{\vp_1})$; \ $(\GR{G}{2}, \VV_{\vp_1})$.
\\[.6ex] For the first three cases, the irreducibility is proved in 
\cite{fan}. So, it remains to handle the last one.

The commuting variety of $\VV$ is determined by the tangent spaces to all $G$-orbits in 
$\VV$, since $(x,y)\in \fe(\VV)$ if and only if $y\in (\g{\cdot}x)^\perp$.
It is known that the $\GR{G}{2}$-orbits in the 7-dimensional module
$\VV_{\vp_1}$ are the same as $SO_7$-orbits. But the commuting
variety for $(SO(\VV),\VV)$ is irreducible for any $\VV$.
\end{proof}%

Philosophically, the above proof (as well as any case-by-case proof) 
is not satisfactory. One ought to argue as follows:

Our previous results suggest that invariant-theoretic properties of
$(G:\vts)$ are determined by  properties of its simple reduction
$\el=\g(\Pi_s)$. We also know, after Richardson,
that $\fe(\el)$ is irreducible. Therefore, it is reasonable to suggest that
the irreducibility of $\fe(\vts)$ can be deduced 
from that of $\fe(\el)$. That is, one may try to prove directly that
$\ov{G{\cdot}\fe(\el)}=\fe(\vts)$.

\subsection{}
The theory exposed in this article suggest that (almost) all results for the adjoint representations should have analogues for the little adjoint representations.
Furthermore, the adjoint representations in the simply-laced case and the little adjoint
representations in multiply-laced case can be treated simultaneously, if we agree
that in the simply-laced case all the roots are short (hence $\vts=\g$,\  
$\Pi_s=\Pi$, \ 
$W(\Pi_s)=W$, \ $W_l=\{1\}$, etc.)


\begin{thebibliography}{Pa95}

\bibitem{bour}
{\sc N.~Bourbaki.}
"Groupes et alg\`ebres de Lie", Chapitres 4,\,5 et 6,
Paris: Hermann 1975.

\bibitem{alela}
{\rusc A.G.~{E1}lashvili}. {\rus Kanonicheski{\u\i} vid i statsionarnye podalgebry
tochek obwego polozheniya dlya prostykh line{\u\i}nykh grupp Li},
{\rusi Funkts.\ analiz i\  prilozh.} {\rus t}.{\bf 6}, {\rus N0}\,1 (1972), 51--62
(Russian). English translation: 
{\sc A.G.~Elashvili}. 
Canonical form and stationary subalgebras of points of
general position for simple linear Lie groups, {\it Funct. Anal. Appl.}
{\bf 6}\,(1972), 44--53.

\bibitem{polar} {\sc J.~Dadok} and {\sc V.G.~Kac}. Polar representations,
{\it J. Algebra} {\bf 92}\,(1985), 504--524.

\bibitem{wim2} {\sc W.~Hesselink.} Characters of the Nullcone, {\it Math.
Ann.} {\bf 252}(1980), 179--182.

\bibitem{ko63} {\sc B.~Kostant}. Lie group representations in
polynomial rings, {\it  Amer. J. Math.} {\bf 85}\,(1963), 327--404.

\bibitem{luna72} {\sc D.~Luna}. Sur les orbites ferm\'ees des groupes alg\`ebriques
r\'eductifs, {\it Invent. Math.} {\bf 16}\,(1972), 1--5. 

\bibitem{lr} {\sc D.~Luna, R.W.~Richardson}. A generalization of the
Chevalley restriction theorem, {\it Duke Math. J.} {\bf 46}\,(1979), 487--496.

\bibitem{ohta}  {\sc T.~Ohta}. 
The singularities of the closures of nilpotent orbits in certain symmetric pairs,
{\it T\^ohoku Math. J.}, II. Ser. {\bf 38}\,(1986), 441--468.

\bibitem{fin} {\rusc D.I.~Panyushev}. 
{\rus O prostranstvah orbit koneqnyh i sv{j1}znyh line{\u\i}nyh grupp},
{\rusi Izv. AN SSSR. Ser. matem.} {\rus t}.{\bf 46}, {\rus N0}\,1 (1982), 95--99 (Russian).
English translation: {\sc D.~Panyushev}.
On orbit spaces of finite and connected linear groups,
{\it Math. USSR-Izv.} {\bf 20}\,(1983), 97--101.

\bibitem{R&S0} {\rusc D.I.~Panyushev}.
{\rus Regul{j1}rnye {e1}lement{y} v prostranstvah line{\u\i}nyh 
predstav\-le\-ni{\u\i}  reduktivnyh algebraiqeskih grupp},
{\rusi Izv. AN SSSR. Ser. matem.} {\rus t}.{\bf 48}, {\rus N0}\,2 (1984), 411--419 (Russian).
English translation: {\sc D.~Panyushev}. 
Regular elements in spaces of linear representations of reductive algebraic groups,
{\it Math. USSR-Izv.} {\bf 24}\,(1985), 383--390.

\bibitem{R&S} {\rusc D.I.~Panyushev}.
{\rus Regul{j1}rnye {e1}lement{y} v prostranstvah line{\u\i}nyh 
predstav\-le\-ni{\u\i} } II, {\rusi Izv. AN SSSR. Ser. matem.},
{\rus t}.{\bf 49}, {\rus N0}\,5 (1985), 979--985 (Russian). 
English translation: {\sc D.~Panyushev}. Regular elements in spaces of linear 
representations II, {\it Math. USSR-Izv.}, {\bf 27}\,(1986), 279--284.

\bibitem{ya-tg} {\sc D.~Panyushev}. The exterior algebra and ``spin''
of an orthogonal $\g$-module, {\it Transformation Groups},
{\bf 6},~no.\,4 (2001), 371--396.

\bibitem{short04} {\sc D.~Panyushev}.
Short antichains in root 
systems, semi-Catalan arrangements, and $B$-stable subspaces,
{\it Europ. J. Combin.}, {\bf 25}\,(2004), 93--112.

\bibitem{fan}  {\rusc D.I.~Panyushev}. 
{\rus O neprivodimosti kommutatornyh mnogoobrazi{\u\i}, svyazannyh s
involyuciyami prostyh algebr Li},  {\rusi Funkc. analiz i  prilo{z1}.}
{\rus t}.{\bf 38}, {\rus vyp.}\,1~(2004), 47--55 (Russian). English translation: 
{\sc D.\,Panyushev}. On the irreducibility of commuting
varieties associated with involutions of simple Lie algebras,
{\it Funct. Anal. Appl.} {\bf 38}\,(2004), 38--44.

\bibitem{selecta} {\sc D.~Panyushev}.
Generalised Kostka-Foulkes  polynomials and cohomology of line bundles on
homogeneous vector bundles, {\it Selecta Math., New Ser.}, {\bf 16}\,(2010), 315--342.

\bibitem{po} {\rusc V.L.~Popov}. {\rus Kriteri{\u\i} stabil\cprime nosti de{\u\i}stviya
poluprosto{\u\i} gruppy na faktorial{\cprime}nom mno\-go\-ob\-ra\-zii},
{\rusi Izv. AN SSSR. Ser. matem.},
{\rus t}.{\bf 34}, {\rus N0}\,3 (1970), 523--531 (Russian). 
English translation: {\sc V.L.~Popov}. Stability criteria for the action of a semisimple
group on a factorial manifold, 
 {\it Math. USSR-Izv.}, {\bf 4}, no.\,3 (1970), 527--535.

\bibitem{rich}
{\sc R.~Richardson}. Commuting varieties of semisimple Lie algebras
and algebraic groups, {\it Compositio Math.} {\bf 38}\,(1979), 311--327.

\bibitem{sam} {\sc H.~Samelson}. ``Notes on Lie Algebras''
(Universitext). Springer-Verlag, New York, 1990. xii+162 pp.

\bibitem{sek84}  
{\sc J.~Sekiguchi}. The nilpotent subvariety of the vector space associated to a symmetric
pair, {\it Publ. R.I.M.S. Kyoto Univ.} {\bf 20}\,(1984), 155--212.

\bibitem{vi79}
{\rusc {E1}.B.~Vinberg.} {\rus Klassifikatsiya odnorodnykh  nil{\cprime}potentnykh
{e1}lementov poluprosto{\u\i} gra\-du\-iro\-vanno{\u\i} algebry Li, \ V sb.:}
{\rusi "Trudy seminara po vekt. 
i tenz. analizu"}, {\rus t.}\,19, {\rus str.~155--177. Moskva: MGU} 1979
(Russian). 
English translation: {\sc E.B.~Vinberg}. 
Classification of homogeneous nilpotent elements of a semisimple graded Lie algebra,
{\it Selecta Math. Sov.}, {\bf 6}\,(1987), 15--35.

\bibitem{VP}
{\rusc {E1}.B.~Vinberg, V.L.~Popov}.
 {\rusi ``Teoriya Invariantov"}, {\rus V kn.:
Sovremennye problemy matematiki. Fundamentalnye napravleniya, t.\,55,
str.}\,137--309. {\rus Moskva: VINITI} 1989 (Russian).
English translation:
{\sc V.L.~Popov} and {\sc E.B.~Vinberg}. ``Invariant theory", In: {\it
Algebraic Geometry IV}
(Encyclopaedia Math. Sci., vol.~55, pp.123--284)
Berlin Heidelberg New York: Springer 1994.

\bibitem{zarhin}
{\sc Yu.G.~Zarhin}.   Linear simple Lie algebras and
ranks of operators. In: ``The Grothendieck Festschrift", Vol. III, 481--495,
Progr. Math., {\bf 88},  Birkh\"auser Boston, Boston, MA, 1990. 

\end{thebibliography}
\end{document}